\documentclass[12pt,a4paper]{article}
\usepackage{amssymb,amsmath,amsthm}
\usepackage{graphicx}
\usepackage{hyperref}

\usepackage{authblk}
\usepackage{amssymb,amsmath,amsthm}
\usepackage{graphicx}
\usepackage{hyperref}

\usepackage{authblk}
\usepackage{amssymb,amsmath,amsthm}
\usepackage{graphicx}
\usepackage{hyperref}

\usepackage[ruled,vlined]{algorithm2e}
\usepackage{algorithmic}

\usepackage{authblk}
\usepackage{wasysym}

\usepackage{spverbatim} 

\usepackage{array} 
\newcolumntype{P}[1]{>{\centering\arraybackslash}p{#1}}
\newcolumntype{M}[1]{>{\centering\arraybackslash}m{#1}}

\newcommand\cA{{\mathcal A}}

\newcommand\cC{{\mathcal C}}
\newcommand\cD{{\mathcal D}}

\newcommand\cF{{\mathcal F}}

\newcommand\cG{{\mathcal G}}
\newcommand\cH{{\mathcal H}}
\newcommand\cI{{\mathcal I}}

\newcommand\cM{{\mathcal M}}

\newcommand\cT{{\mathcal T}}

\theoremstyle{plain}
\newtheorem{theorem}{Theorem}[section]
\newtheorem{lemma}[theorem]{Lemma}
\newtheorem{corollary}[theorem]{Corollary}
\newtheorem{conjecture}[theorem]{Conjecture}
\newtheorem{proposition}[theorem]{Proposition}
\newtheorem{observation}[theorem]{Observation}

\theoremstyle{definition}

\newtheorem{definition}[theorem]{Definition}

\newtheorem{claim}[theorem]{Claim}

\newtheorem{question}{Question}
\newtheorem*{rem}{Remark}

\newcommand\cref[1]{Corollary~\ref{cor:#1}}

\newcommand{\sat}{{\rm sat}}
\newcommand{\sats}{{\rm sat}^*}

\mathcode`l="8000
\begingroup
\makeatletter
\lccode`\~=`\l
\DeclareMathSymbol{\lsb@l}{\mathalpha}{letters}{`l}
\lowercase{\gdef~{\ifnum\the\mathgroup=\m@ne \ell \else \lsb@l \fi}}%
\endgroup

\title{Induced and non-induced poset saturation problems}
\author{Bal\'azs Keszegh$^{1,2}$, Nathan Lemons$^3$, Ryan R. Martin$^4$,
	~~~~~~~~~~~~~~~~~~~~~~~~~D\"om\"ot\"or P\'alv\"olgyi$^2$, Bal\'azs Patk\'os$^{1,5}$\\ 
\small $^1$ Alfr\'ed R\'enyi Institute of Mathematics \\
\small $^2$ MTA-ELTE Combinatorial Geometry Research Group \\
\small $^3$ Los Alamos National Laboratory \\
\small $^4$ Iowa State University \\
\small $^5$ Lab. of Combinatorial and Geometric Structures, Moscow Inst. of Physics and Technology \\
\small $\{$keszegh,patkos@renyi.hu, nlemons@lanl.gov, dom@cs.elte.hu, rymartin@iastate.edu$\}$}

\date{}

\begin{document}

\maketitle

\begin{abstract}
	A subfamily $\cG\subseteq \cF\subseteq 2^{[n]}$ of sets is a non-induced (weak) copy of a poset $P$ in $\cF$ if there exists a bijection $i:P\rightarrow \cG$ such that $p\le_P q$ implies $i(p)\subseteq i(q)$. In the case where in addition $p\le_P q$ holds if and only if $i(p)\subseteq i(q)$, then $\cG$ is an induced (strong) copy of $P$ in $\cF$. We consider the minimum number $\sat(n,P)$ [resp.\ $\sats(n,P)$] of sets that a family $\cF\subseteq 2^{[n]}$ can have without containing a non-induced [induced] copy of $P$ and being maximal with respect to this property, i.e., the addition of any $G\in 2^{[n]}\setminus \cF$ creates a non-induced [induced] copy of $P$.

	We prove for any finite poset $P$ that $\sat(n,P)\le 2^{|P|-2}$, a bound independent of the size $n$ of the ground set. For induced copies of $P$, there is a dichotomy: for any poset $P$ either $\sats(n,P)\le K_P$ for some constant depending only on $P$ or $\sats(n,P)\ge \log_2 n$. We classify several posets according to this dichotomy, and also show better upper and lower bounds on $\sat(n,P)$ and $\sats(n,P)$ for specific classes of posets. 
	
	Our main new tool is a special ordering of the sets based on the colexicographic order. It turns out that if $P$ is given, processing the sets in this order and adding the sets greedily into our family whenever this does not ruin non-induced [induced] $P$-freeness, we tend to get a small size non-induced [induced] $P$-saturating family.
	\end{abstract}

\section{Introduction}
A subposet $Q'$ of $Q$ is a \textit{weak} or \textit{non-induced copy} of the poset $P$ in $Q$, if there exists a bijection $i:P\rightarrow Q'$ with $p\le_P p'$ implying $i(p)\le_Q i(p')$. In the case where in addition $p\le_P p'$ holds if and only if $i(p)\le_Q i(p')$, then we say that $Q'$ is a \textit{strong} or \textit{induced copy} of $P$ in $Q$. If $Q$ does not contain a non-induced [induced] copy of $P$, then we say that $Q$ is non-induced [induced] $P$-free. The extremal forbidden subposet problem asks for the maximum size of a non-induced [induced] $P$-free subposet of $Q$. To generalize results of Sperner \cite{S} and Erd\H os \cite{Erd}, this was introduced by Katona and Tarj\'an \cite{KT} in the case where $Q=B_n$ is the poset of all subsets of an $n$-element set ordered by inclusion. It is conjectured (implicitly in the work of Katona and his co-authors, explicitly by Bukh \cite{Bukh}, and Griggs and Lu \cite{GLu}) 
that the size of a maximum poset divided by ${{\binom {n} {\lfloor n/2 \rfloor}}}$ always tends to the size of a maximum number of complete and consecutive middle levels of the Boolean lattice whose union is $P$-free (in both the non-induced and induced cases), but this has been verified only in special cases; the fact that this limit is bounded follows from \cite{Erd} and \cite{MP}, respectively.
For more on this topic see the recent survey \cite{GL} and Chapter 7 of \cite{GP}.

The corresponding saturation problem asks for the minimum possible sizes, denoted $\sat(Q,P)$ $[\sats(Q,P)]$, of a non-induced [induced] $P$-free subposet of $Q$ that is maximal with respect to being $P$-free. Such subposets are said to be non-induced [induced] $P$-saturating and in case $Q=B_n$, we write $\sat(n,P)$ and $\sats(n,P)$. First, Gerbner et al.~\cite{G6} studied this problem for $P=C_k$, the chain on $k$ elements, in which case  $\sat(n,C_k)=\sats(n,C_k)$. They proved that $2^{(k-3)/2}\le\sat(n,C_{k})\le 2^{k-2}$ holds for all $n$, and showed that for $k=7$ the upper bound can be further strengthened to $\sat(n,C_7)\le 30$. This latter upper bound was generalized by Morrison, Noel, and Scott \cite{MNS}, proving $C\cdot 2^{(1-\delta)k}$ where $\delta=1-\frac{\log_2 15}{4}\approx 0.02$. Later, the induced version $\sats(n,P)$ was studied by Ferrara et al.~\cite{F7} and by Martin, Smith, and Walker \cite{MSW}.
In \cite{F7}, it was shown for a number of other posets $P$ that $\sat(n,P)$ is bounded by some constant independent of $n$, while $\sats(n,P)$ was shown to be unbounded for all these posets. Ivan \cite{I} has very recently improved lower bounds on $\sats(n,\bowtie)$ and $\sats(n,N)$ for the butterfly and the $N$-posets.

Our first main result proves that non-induced saturation numbers are always bounded by a function of $|P|$, which is a constant independent of $n$.

\begin{theorem}\label{bounded}
For any finite poset $P$, we have $\sat(n,P)\le 2^{|P|-2}$.
\end{theorem}

Unlike in the case of the extremal forbidden subposet problem,  the proof of Theorem \ref{bounded} does not follow from the fact that $\sat(n,C_k)$ is bounded.

Note that as shown by the result $2^{(k-3)/2}\le\sat(n,C_{k})$ of \cite{G6}, an exponential rate of growth in $k$ is best possible. One might wonder which $k$-element poset is hardest to saturate.

\begin{conjecture}\label{chainisbest}
For any $k$-element poset $P$, we have $\sat(n,P)\le\sat(n,C_k)$.
\end{conjecture}
 
Section 2 contains the proof of Theorem \ref{bounded} along with specific better bounds for several poset classes.

Then we move on to induced saturation problems. 
We prove the following dichotomy result, which is implicitly contained in \cite{F7}.

\begin{theorem}\label{dich}
	For any poset $P$, either there exists a constant $K_P$ with $\sats(n,P)\le K_P$ or for all $n,\;\sats(n,P)\ge \log_2 n$.
\end{theorem}

We conjecture that the following strengthening of Theorem \ref{dich} holds.

\begin{conjecture}
	For any poset $P$, either there exists a constant $K_P$ with $\sats(n,P)\le K_P$ or for all $n,\; \sats(n,P)\ge n+1$.
\end{conjecture}

In Section 3, we prove Theorem \ref{dich} and a number of lower and upper bounds on $\sats(n,P)$ for several classes of posets $P$. In particular, with a new construction and with the recent lower bound by Ivan \cite{I}, we establish $\sats(n,\bowtie)=\Theta(n)$.

A collection of bounds can be found in Table \ref{table}.
We list the best known bounds for all posets on at most $4$ elements and some further general results.
For the induced results we marked in which paper or statement the proof can be found, while the non-induced results follow from Proposition \ref{max1}, or a short case analysis.\\

\textbf{Notation.} As we work in $B_n$, the poset of all subsets of an $n$-element set ordered by inclusion, we will speak about families $\cF$ of subsets and we will say that these families are non-induced/induced $P$-saturating if so are the corresponding subposets of $B_n$. We use the standard notation $[n]$ for the set of the first $n$ positive integers and $2^X$ for the power set of $X$.
We say that two elements $x,y\in [n]$ are \emph{separated} by $\cF$ if there is an $F\in \cF$ such that $|F\cap \{x,y\}|=1$.
The family $\cF$ is \emph{separating} if any two elements of $[n]$ are separated by $\cF$.

For two posets, $P$ and $Q$, we denote their (incomparable) disjoint union by $P+Q$. We denote $P+\dots +P$, the disjoint union of $k$ copies of $P$, by $kP$.
We denote the chain on $k$ elements by $C_k$, and the antichain of $k$ elements by $A_k$, i.e., $A_k=kC_1$.
For an arbitrary poset $P$, we denote by $\dot P$ the poset obtained from $P$ by adding to it an element that is larger than all elements of $P$.

The poset on $4$ elements in which two incomparable elements are both below two other elements that are incomparable is called the \emph{butterfly} poset denoted by $\bowtie$. 

Throughout the paper $\log$ stands for the logarithm in base 2.

\begin{table}
	\centering
	\setlength{\tabcolsep}{10pt}
	\begin{tabular}{| c | l | l l |}
		\hline
		{\bf poset} $P$ 	& ${\bf sat}(n,P)$  	& \multicolumn{2}{l|}{${\bf sat}^*(n,P)$} \\ \hline
		
		
		$C_2$, chain & $= 1$ & \multicolumn{2}{l|}{$= 1$} \\ \hline
		
		$A_2$, antichain & $= 1$ & \multicolumn{2}{l|}{$= n+1$} \\ \hline
		
		$C_3$, chain & $=2$ & \multicolumn{2}{l|}{$=2$} \\ \hline
		
		$C_2+C_1$, chain and single & $=2$ & $=4$ & case analysis
		\\ \hline
		
		$\vee$ fork (or $\wedge$) & $=2$ & $=n+1$ & [F7] \\ \hline	
		
		$A_3$, antichain & $=2$ & $=3n-1$ & [F7] \\ \hline
		
		$C_4$, chain & $=4$  & $=4$ & [G6] \\ \hline
		
		$\vee_3$, fork with three tines & $=3$ & $\geq \log_2 n$ & [F7] \\ \hline	
		
		$\Diamond$, diamond 	& $=3$ 	& $\geq\sqrt{n}$ & [MSW] \\
		& 		& $\le n+1$ & [F7] \\ \hline
		
		$\Diamond^-$, diamond minus an edge & $=3$ & $=4$ & case analysis
		\\ \hline	
		
		$\Bowtie$, butterfly 	& $=4$ 	& $\ge n+1$ 	& [I] \\
		& 		& $\le 6n-10$ 	& [Thm \ref{butterfly}] \\ \hline
		
		Y & $=3$ & $\ge \log_2 n$ & [Thm. \ref{thm:uctp+chain}] \\ \hline
		
		N 	& $=3$ 	& $\ge \sqrt n$ 	& [I] \\
		& 	& $\le 2n$ 	& [F7]  \\ \hline
		
		$2C_2$ 	& $=3$ 	& $\ge n+2$ 	& [Thm. \ref{2+2}] \\
		& 		& $\le 2n$ 	& [Prop. \ref{2C2UB}] \\ \hline
		
		$C_3+C_1$, chain and single & $=3$ & $\le 8$ & [Prop. \ref{prop:const}]	\\ \hline
		
		$\vee+1$, fork and single & $=3$ & $\ge \log_2 n$ & [F7] \\ \hline
		
		$C_2+A_2$ & $=3$ & $\le 8$ & [Prop. \ref{prop:const}]	\\ \hline
		
		$A_4$, antichain 	& $=3$ 	& $\ge 3n-1$ & [F7] \\
		& 		& $\le 4n+2$ & [F7] \\ \hline

		$C_5$, chain & $=8$	& $=8$ 		& [G6]+[MNS] \\ \hline
		
		$C_6$, chain & $=16$	& $=16$ 		& [G6]+[MNS] \\ \hline
		
		$C_k$, chain ($k\ge 7$) & $\ge 2^{(k-3)/2}$ 	& $\ge 2^{(k-3)/2}$ 	& [G6] \\
		& $\le 2^{0.98k}$& $\le 2^{0.98k}$ 	& [MNS] \\ \hline
		
		$A_k$, antichain & $=k-1$	& $\ge \left(1-\frac{1}{\log_2 k}\right)\frac{k}{\log_2 k}n$ 	& [MSW] \\
		& 	& $\le kn-k-\frac 12\log_2 k+O(1)$ \hspace{-6mm} & [F7] \\ \hline
		
		$3C_2$ &	$=5$		&   $\le 14$ 			& [Prop. \ref{prop:3C2}] \\ \hline

		$5C_2$ 	&	$=9$		&   $\le 42$ 			& [Prop. \ref{prop:const}] \\ \hline
		
		$7C_2$ 	&	$=13$	&   $\le 60$ 			& [Prop. \ref{prop:const}] \\ \hline		
		
		%
		%
		any poset on $k$ elements & $\le 2^{k-2}$ & --- & [Thm.~\ref{bounded}] \\ \hline
		UCTP \footnotesize{(def.~in Section \ref{sec:uctp})}  & $O(1)$ & $\ge \log_2 n$  & [F7]   \\ \hline
		UCTP with top chain & $O(1)$& $\ge \log_2 n$ & [Thm. \ref{thm:uctp+chain}]   \\ \hline
		chain + shallower & $O(1)$ &  $O(1)$ & [Thm.~\ref{thm:chain+shallow}] \\ \hline

	\end{tabular}
	\caption{Summary of all posets with at most 4 elements as well as some additional examples and categories.
		\footnotesize F7: Ferrara et al.~\cite{F7}; G6: Gerbner et al.~\cite{G6}; I: Ivan~\cite{I}; MNS: Morrison, Noel, Scott~\cite{MNS}; MSW: Martin, Smith, Walker~\cite{MSW}.}
	\label{table}
\end{table}

\section{Non-induced results}
In this section we only consider non-induced results, therefore we omit this adjective throughout the section.

Note that if $P$ is a poset on $k$ elements, then $\sat(n,P)\ge k-1$ trivially holds if $n$ is big enough.
We show that this bound is often tight, due to the non-induced nature of the problem.

\begin{proposition}\label{max1}
	If $P$ is a poset on $k$ elements,	
	and $\exists p\in P$ such that there is at most one larger and at most one smaller element than $p$ in $P$, i.e., $|\{q:q<_Pp\}|,|\{q':p<_Pq'\}|\le 1$, then $\sat(n,P)=k-1$ for $n\ge k$.
\end{proposition}

\begin{proof}
	Let $p^+,p^-\in P\setminus \{p\}$ be the only elements with $p^-<p<p^+$ (if they exist). Then, by definition, $p^+$ is maximal and $p^-$ is minimal in $P$. Let $m$ be the smallest integer such that $Q_m$ contains a copy of $P\setminus \{p\}$, and let $i:P\setminus \{p\} \rightarrow 2^{[m]}$ be an embedding showing this. Note that $m\le k-2$ as $Q_{k-2}$ contains a copy of $C_{k-1}$ which in turn contains a copy of any other poset of size at most $k-1$. Let $i_n:P\setminus \{p\} \rightarrow 2^{[n]}$ be defined as $i_n(p^-)=\emptyset$, $i_n(p^+)=[n]$, and $i_n(p')=i(p)\cup \{m+1\}$ for any $p'\notin \{p,p^+,p^-\}$. If $n\ge m+2$, then $i_n$ is a bijection and its image is a copy of $P\setminus \{p\}$ (here we use the fact that $p^+$ is maximal and $p^-$ is minimal in $P$).  Clearly any $F\notin i_n[P\setminus\{p\}]$ extends the image of $i_n$ to a copy of $P$ as $i_n(p^-)=\emptyset \subset F \subset [n]=i_n(p^+)$ and so $F$ is a suitable image of $p$.	
\end{proof}

Before getting to the more involved proofs, let us state another simple observation.

\begin{proposition}\label{bk-2}
	If $P$ is a poset on $k$ elements and $B_k$ is $P$-free, then $sat(n,P)\le 2^{k-2}$.
\end{proposition}

\begin{proof}
Let $f:2^{[k-2]} \rightarrow 2^{[n]}$ be defined as $f(F)=F$ if $k-2\notin F$ and $f(F)=F \cup ([n]\setminus [k-2])$ otherwise. Then $f$ is a poset isomorphism between $B_{k-2}$ and $f(2^{[k-2]})=:\cF$, so $\cF$ is $P$-free. On the other hand, $\cF$ is the $C_k$-saturating family from \cite{G6}, so for any $G\notin \cF$, the family $\cF\cup \{G\}$ contains a copy of $C_k$ and thus a copy of $P$.
\end{proof}

\subsection{Proof of Theorem \ref{bounded}}

Let us define the colexicographic ordering (or \textit{colex ordering}) on all finite subsets of $\mathbb{Z}^+$ as usual by $A<B$ if and only if $\max A\triangle B\in B$ holds, where $A\triangle B$ is the symmetric difference $(A\setminus B)\cup (B\setminus A)$. Let $P$ be any poset on $k$ elements. Let $n\ge k$ and let $F_1,F_2,\dots,F_{2^{n-1}}$ be the enumeration of $2^{[n-1]}$ (the sets not containing $n$) in colex order, let $m_i=\max F_i$, and let $G_i=[n]\setminus F_i$ for every $1\le i \le 2^{n-1}$ (the sets containing $n$). Note that every subset of $[n]$ is either enumerated as an $F_i$ or as a $G_i$. Let us consider the \emph{greedy colex process} that tries to add these sets in order (see Algorithm \ref{algo}).

Theorem \ref{bounded} is an immediate consequence of the following.

\begin{theorem}\label{greedy}
For $1\leq k\leq n$, let $P$ be a $k$-element poset and let $\cF:=\cF_{2^{n-1}}$ be the output of the greedy colex process (as defined in Algorithm \ref{algo}). Then, $\cF$ is $P$-saturating, $\cF=\cF_{2^{k-3}}$ and therefore $|\cF|\le 2^{k-2}$. In particular, $\sat(n,P)\le 2^{k-2}$ holds.
\end{theorem}

\begin{proof}
The fact that $\cF$ is $P$-saturating (that it is both $P$-free and the addition of any element forms a copy of $P$) is clear from the definition of the greedy process. Observe that if $j<i$, then $F_i \not\subseteq F_j$ and consequently $G_j \not\subseteq G_i$. Finally, note that $G_j\not\subset F_i$ for all $i,j$.  Based on this we have the following lemma.

\begin{lemma}\label{colex} For any $i\le 2^{n-1}$ we have the following.
\begin{enumerate}
\item
$F_i\in \cF_i$ implies $F_i\setminus \{m_i\}\in \cF_j$ for some $j< i$.
\item
$F_i \in \cF_i$ implies $F_i \cup ([n]\setminus [m_i]) \in \cF_j$ for some $j< i$.
\item
$G_i \in \cF_i$ implies $G_i \cup \{m_i\}\in \cF_j$ for some $j< i$.
\item
$G_i \in \cF_i$ implies $G_i \cap [m_i]\in \cF_j$ for some $j< i$.
\end{enumerate}
\end{lemma}

\begin{algorithm}\label{algo}
	\caption{Greedy colex process}
	\begin{algorithmic}
		\STATE Set $\cF_0=\emptyset$
		
		\FOR {$i<2^{n-1}$}
		\IF {$\cF_i\cup \{F_{i+1}\}$ is $P$-free}
		\STATE $\cF'_i:=\cF_i\cup \{F_{i+1}\}$
		\ELSE \STATE{$\cF'_{i}:=\cF_i$}
		\ENDIF
		\IF {$\cF_i' \cup \{G_{i+1}\}$ is $P$-free}
		\STATE $\cF_{i+1}:=\cF'_i\cup \{G_{i+1}\}$
		\ELSE \STATE {$\cF_{i+1}:=\cF'_i$}
		\ENDIF
		\ENDFOR
		
		Output $\cF_{2^{n-1}}$
	\end{algorithmic}
\end{algorithm}	

\begin{proof}
To see (1), let $j$ be defined such that $F_j=F_i\backslash\{m_i\}$ and observe that $j<i$. We claim that for any $H\in \cF_{j-1}$, the pair ($H,F_j$) has the same relations as the pair ($H,F_i$).  
If $H$ was enumerated as an `$F$', i.e., $n\notin H$, then $H < F_j$ means that $H$ contains neither $F_j$ nor $F_i$.  In addition, such an $H$ must be a subset of $[m_j]$, and therefore, $H\subset F_j \iff H\subset F_i$.  
Otherwise, $H$ was enumerated as a `$G$', i.e., $n\in  H$, and thus contains $[n]\setminus [m_j]$.  In particular, $H$ is contained in neither $F_j$ nor $F_i$ and $F_j\subset H \iff F_i\subset H$.
Therefore, as $F_i$ is included in $\cF_i$ because its addition did not create a copy of $P$ in $\cF_i$, the addition of $F_j$ also does not create a copy of $P$ in $\cF_j$, and so it was added to the family.
Thus we must also have $F_i \setminus \{m_i\}=F_j\in \cF_j\subseteq \cF_i$. 

As the proofs of the other statements are similar, we just sketch them. To see (2), observe that $[n]\setminus (F_i \cup ([n]\setminus [m_i])) \subseteq [m_i-1]$, thus $F_i \cup ([n]\setminus [m_i])=G_j$ for some $j<i$. It is left to the reader to check that for any $H\in \cF_j'$ the containment relation of the pair ($H,G_j$) is the same as that of the pair ($H,F_i$).

To see (3), observe that $G_i \cup \{m_i\}=G_j$ for some $j<i$. It is left to the reader to check that for any $H\in \cF'_j$, the containment relation of the pair ($H,G_j$) is the same as that of ($H,G_i$).

Finally, to see (4), observe that $G_i \cap [m_i] \subseteq [m_i-1]$, thus $G_i \cap [m_i]$ is $F_j$ for some $j<i$. It is left to the reader to check that for any $H\in \cF_{j-1}$ the containment relation of the pair ($H,F_j$) is the same as that of the pair ($H,G_i$).
\end{proof}

Let us return to the proof of Theorem \ref{greedy}. Towards a contradiction, suppose that there exists some $H\in \cF\setminus \cF_{2^{k-3}}$. We distinguish two cases.

\vskip 0.2truecm

\textsc{Case I.} $H=F_i$ for some $i> 2^{k-3}$.

\vskip 0.15truecm

Then write $F_i=\{h_1,h_2,\dots,h_l\}$ with $h_1<h_2<\dots<h_l=m_i$, where $m_i>k-3$. Repeated applications of Lemma \ref{colex} (1) imply that $H_r=\{h_1,h_2, \dots, h_r\}\in \cF_i$ holds for all $r=l,l-1,\dots,1,0$, giving a decreasing chain of length $l+1$ contained in $\cF$.
By Lemma \ref{colex} (2), $G_j=F_i\cup([n]\backslash[h_l])\supsetneq F_i$ and $G_j \in \cF$.
Repeated applications of Lemma \ref{colex} (3), starting with $G_j$, give an increasing chain of length $m_i-l+1$ contained in $\cF$.
Putting the two chains together gives a chain of length $m_i+2\ge k$ contained in $\cF$, a contradiction, because such a chain (and thus also $\cF$) contains a (non-induced) copy of $P$.

\vskip 0.2truecm

\textsc{Case II.} $H=G_j$ for some $j> 2^{k-3}$.

\vskip 0.15truecm

We apply Lemma \ref{colex} (4) to $G_j$ to get an $F_i\subsetneq G_j$ such that $F_i\in \cF$. Then we get a chain of length $m_i+2\ge k$ as the union of two chains as in the previous case by repeated applications of Lemma \ref{colex} (1) and (3).
\end{proof}

One might notice that in this proof we have only used that a given family of subsets of $[n]$ that is saturating for $C_k$ (namely the sets in $[k-3]$ and their complements) could be ordered in a suitable way using the colex ordering.
If this was true for other saturating families for $C_k$, then that would prove Conjecture \ref{chainisbest}.
However, the families achieving the current best bound, $\sat(n,C_k)\le O(2^{0.98k})$ \cite{MNS}, cannot be ordered like that in a straightforward way.

\subsection{Complete bipartite posets}
In this subsection we consider complete bipartite posets $K_{s,t}$ on $s+t$ elements $a_1,\dots,a_s$, $b_1,\dots,b_t$ with $a_i<b_j$ for any $1\le i\le s$, $1\le j\le t$. Observe that if $s$ or $t$ equals 1, then Proposition \ref{max1} yields $\sat(n,K_{1,t})=t$ and $\sat(n,K_{s,1})=s$.
\begin{proposition}\label{kst}
	If $s,t\ge 2$ and $n\ge s+t-3$ hold, then $s+t-1\le \sat(n,K_{s,t})\le 2(s+t)-4$. 
\end{proposition}

\begin{proof}

	The lower bound follows from $\sat(n,P)\ge |P|-1$. For the upper bound, let $\cF_0=\{\emptyset, E_1,E_2,\dots,E_{s+t-3}\}$ with $|E_i|=1$ for all $1\le i\le s+t-3$. Let $\cF_1=\{[n]\setminus E: E\in \cF_0\}$, and finally let $\cF=\cF_0\cup \cF_1$. First observe that $\cF$ is $K_{s,t}$-free. Indeed, as $s\ge 2$ and sets in $\cF_0$ contain at most one other set from $\cF$, they can only play the role of some $a_i$. Similarly, as $t\ge 2$ and sets in $\cF_1$ are contained in at most one other set from $\cF$, they can only play the role of some $b_j$. So to form a copy of $K_{s,t}$ we would need at least $s-1$ non-empty sets from $\cF_0$ and at least $t-1$ sets other than $[n]$ from $\cF_1$, so by the pigeonhole principle we would need to pick some $E_i$ and $[n]\setminus E_i$ as well, but for these the containment does not hold.
		
		To see that $\cF$ is $K_{s,t}$-saturating, let $G$ be any set from $2^{[n]}\setminus \cF$. If $G$ contains at least $s-1$ $E_i$'s, then these $E_i$'s and $\emptyset$ can form the bottom of $K_{s,t}$. Meanwhile $G$ and the sets of $\cF_1$ that are not complements of these $s-1$ $E_i$'s (there are exactly $1+(t+s-3)-(s-1)=t-1$ of them) can form the top of $K_{s,t}$. Otherwise, $G$ contains at most $s-2$ $E_i$'s and therefore $G$ is contained in the complement of at least $(t+s-3)-(s-2)=t-1$ $E_i$'s.
		By symmetry, we can repeat the previous argument to get that $t-1$ such complements $F_1,\dots, F_{t-1}$, $[n]$, $\emptyset$, $G$, and the $E_j$'s in $\cap_{j=1}^{t-1}F_j$ form a copy of $K_{s,t}$.
\end{proof}

\subsection{Posets of graphs}
In this subsection we consider the following class of posets. Let $G=(V,E)$ be any finite multigraph without loops. Then let us define the poset $P(G)$ on $V\cup E$ such that $v< e$ if and only if $v\in e$, while $V$ and $E$ form two antichains.

\begin{proposition}\label{keszegh}
Let $G$ be a graph with $e$ edges and $v$ vertices and let $n\ge e+v$. 
Then we have $e+v-1\le\sat(n,P(G))\le e+v$.
\end{proposition}

\begin{proof}
The lower bound follows from $|P|-1\le\sat(n,P)$.
If the minimum degree $\delta$ is at most $1$, then we have $\sat(n,P(G))=e+v-1$ from Proposition \ref{max1}.


Suppose from now on that $\delta \ge 2$ (which implies $e\ge 2$). The construction for the upper bound is as follows. Let $\cF_0=\{\emptyset, E_1,E_2,\dots,E_{v-1},K\}$ with $|E_i|=1$ for all $1\le i\le v-1$ and  $K=[n]\setminus \cup_{i=1}^{v-1}E_i$. Let us define $G_i=[n]\setminus E_i$ for all $1\le i\le v-1$. Let $H_1,H_2,\dots,H_{e-\delta}$ be sets of size $n-1$ all containing $\cup_{i=1}^{v-1}E_i$. Put $\cF_1=\{[n],G_1,\dots,G_{\delta-2},H_1,H_2,\dots,H_{e-\delta}\}$ and consider $\cF=\cF_0\cup \cF_1$. Observe that $|\cF|=e+v=|P(G)|$ and all sets in $\cF_0$ contain at most one other set in $\cF$, so in a copy of $P$ only the $e-1$ sets in $\cF_1$ could play the role of edges of $G$ and thus $\cF$ is $P(G)$-free. Let $F$ be any set not in $\cF$. We need to show that $\cF\cup \{F\}$ contains a copy of $P(G)$. We distinguish two cases.

If $F\subsetneq K$, then $F$ is contained in $K,[n]$ and $G_1,G_2,\dots, G_{\delta-2}$, so $F$ can play the role of a degree $\delta$ vertex in $G$, and $K,[n]$ and $G_1,G_2,\dots, G_{\delta-2}$ can play the role of the edges incident to that vertex. Then, for each $1\le i\le \delta-2$, $E_{i+1}$ can play the role of the other endvertex of the edge corresponding to $G_i$ and $E_1$ the role of the other endvertex of the edge corresponding to $[n]$ and $\emptyset$ can play the role of the other endvertex of the edge corresponding to $K$. Finally, as all $H_j$'s contain all $E_i$'s, this can be easily extended to a copy of $P(G)$. (Note that if $\delta=2$, then there are no sets of the form $G_i$, but the proof works.)

Otherwise, suppose $F\not\subseteq K$ or equivalently $F$ contains at least one of the $E_i$'s, say $E_{1}$. Then $F$ can play the role of an edge $e^*\in E(G)$, and $E_{1}$ and $\emptyset$ can play the role of the two end-vertices of $e^*$. We can choose arbitrarily which endvertex of $e^*$ corresponds to $\emptyset$.

Observe that there is no vertex incident to all edges as otherwise we would have $\delta=1$. Also, if there exists a vertex incident to all but one edge, then either $\delta=1$ or the graph is the triangle for which one can check $\cF$ is saturating. This means we can assume that there are two disjoint edges in $E(G)\setminus \{e^*\}$.

Now one can let the remaining $v-2$ $E_i$'s play the role of the other vertices of $G$ arbitrarily. If $e=2$, then $[n]$ can play the role of the other edge and we are done. 
Otherwise, we need to define a mapping from $\{e:e\in E(G)\setminus e^*\}$ to $\cF_1$ such that if $e=u_\alpha u_\beta$ and $E_\alpha,E_\beta$ play the role of $u_\alpha,u_\beta$, then $e$ is mapped neither to $[n]\setminus E_\alpha$ nor to $[n]\setminus E_\beta$. Consider the auxiliary bipartite graph with parts $\{e:e\in E(G)\setminus e^*\}$ and $\cF_1$ such that for an edge $e=u_\alpha u_\beta\in E(G)$, the vertex $e$ is connected to $S\in \cF_1$ if and only if $E_\alpha \cup E_\beta$ is contained in $S$. By the observations in previous paragraph, we have $|N(e)|\ge \max\{|\cF_1|-2,1\}$ for any $e$, $|N(e)\cup N(e')|\ge |\cF_1|-1$ for any $e,e'$ and $|N(e)\cup N(e')|\ge |\cF_1|$ for any non-adjacent $e,e'$ and $|\cup_{e\in E(G)\setminus \{e^*\}}N(e)|=|\cF_1|$. 
These imply that Hall's condition holds and so we can match the $e_i$'s with sets of $\cF_1$. This gives a copy of $P(G)$ in $\cF\cup \{F\}$ as required.
\end{proof}

One might wonder whether the statement of Proposition \ref{keszegh} remains valid if we allow $G$ to be a multigraph. We do not know, but the construction above does not necessarily work because, if $e$ and $e'$ are parallel edges, then in the above reasoning we can have $|N(e)|=|N(e')|=|N(e)\cup N(e')|=|\cF_1|-2$. In particular, the construction above is not saturating if $G$ consists of 3 parallel edges on 2 vertices. Another problem that can occur is that for multigraphs we can have $\delta-2\ge |V|$, so we would not be able to introduce $G_{\delta-2}$.




Note that if $G$ is the cycle of length $k$, then $P(G)$ is the generalized crown poset on $2k$ elements $a_1,a_2,\dots,a_k,b_1,b_2,\dots,b_k$ with $a_i<b_i,b_{i+1}$ for $i<k$ and $a_k<b_k,b_1$. 
The special case $k=2$ gives the multigraph on two vertices with two edges, for which $P(G)$ is the so-called butterfly poset.
It is easy to check that in this case the upper bound gives the correct answer.



\section{Induced results}
In this section we only consider induced results, therefore we omit this adjective throughout the section.
We start with a simple observation that is useful to determine $\sats(n,P)$ exactly for small values.

\begin{observation}\label{obs}
	If $P$ has no largest element, then any $P$-saturating family must contain the full set, $[n]$, and any such family is also automatically saturating for $\dot P$.
	Similarly, if $P$ has no smallest element, then any $P$-saturating family must contain $\emptyset$.
\end{observation}

We remark that possibly the following stronger statement also holds, but we could not verify either direction, except in special cases (see Theorem \ref{thm:uctp+chain}).

\begin{conjecture}
	$\sats(n,P)$ is bounded if and only if  $\sats(n,\dot P)$ is bounded.
\end{conjecture}

\subsection{Dichotomy}

Call two families over $[n]$ \emph{poset-isomorphic} if there is a bijection between their sets that preserves union, intersection and complement.
Such a bijection is called a \emph{poset-isomorphism}.

\begin{definition}
	For a family $\cF\subseteq 2^{[n]}$, let $\cA(\cF)$ be the algebra it generates using the operations union, intersection, complement.	
	Write $\cF\cong_1 \cF'$ if the following are satisfied.
	\begin{itemize}
		\item $\cF$ and $\cF'$ are isomorphic as posets,
		\item $\cF$ and $\cF'$ have a poset-isomorphism which induces an isomorphism of $\cA(\cF)$ and $\cA(\cF')$,
		\item singletons in $\cA(\cF)$ correspond to singletons in $\cA(\cF')$, i.e., their one element sets are the same.
	\end{itemize}

	Note that the base sets of $\cF$ and $\cF'$ may have different sizes.
	Write $\cF\gtrsim_1 \cF'$ if $\cF$ and $\cF'$ are isomorphic posets, their isomorphism induces an isomorphism of $\cA(\cF)$ and $\cA(\cF')$, and all singleton atoms of $\cA(\cF)$ are also singleton atoms in $\cA(\cF')$.
\end{definition}

For example, consider the following families over the 6-element base set $\{1,...,6\}$:\\ $\cF_1=\{\{1\}, \{1,2\}\}$, $\cF_2=\{\{1\}, \{1,3\}\}$, $\cF_3=\{\{1\}, \{1,2,3\}\}$, $\cF_4=\{\{1\}, \{1,2,3,4\}\}$.\\
We have $\cF_1\cong_1 \cF_2\lesssim_1\cF_3\cong_1\cF_4$, while $\cF_5=\{\{1\}, \{2\}\}$ and $\cF_6=\{\{1,3\}, \{2,3\}\}$ would not be in relation with any of the other families, or with each other.

\begin{proposition}\label{O1G}
	If $\cF\gtrsim_1 \cF'$ and $\cF$ is $P$-saturating, then $\cF'$ is also $P$-saturating.
	Thus, if $\cF\cong_1 \cF'$, then $\cF$ is $P$-saturating if and only if $\cF'$ is $P$-saturating.
\end{proposition}

Note that the converse is not necessarily true; it can happen that $\cF\gtrsim_1 \cF'$ and $\cF'$ is $P$-saturating but $\cF$ is not.
For example, over the 4-element base set $\{1,2,3,4\}$ the family $\cF'=\{\{1\},\{2,3,4\}\}$ is $C_2$-saturating, but $\cF=\{\{1,2\},\{3,4\}\}$ is not, as we can add $\{2,3\}$.

\begin{proof}
	Suppose that $\cF'$ is not $P$-saturating.
	Since $\cF$ and $\cF'$ are isomorphic, $\cF'$ is $P$-free, so the saturation needs to fail, i.e., for some $F'\notin \cF'$ there is no copy of $P$ in $\cF'\cup\{F'\}$.
	Using the isomorphism between $\cF$ and $\cF'$, and $\cA(\cF)$ and $\cA(\cF')$, we can create an $F$ that is in the same relation to the sets of $\cF$ as $F'$ is to the sets of $\cF'$:
	if $F'$ is disjoint from an atom, make $F$ also disjoint from the image of that atom;
	if $F'$ contains an atom, make $F$ also contain the image of that atom; 
	if $F'$ properly cuts into an atom, make $F$ also cut into the image of that atom.
	Here we use that if an atom is non-singleton in $\cA(\cF')$, it is also non-singleton in $\cA(\cF)$.
	Thus, since $\cF'\cup\{F'\}$ is $P$-free, so is $\cF\cup\{F\}$, contradicting the assumption.
\end{proof}

\begin{corollary}\label{cor:const}
	If for a poset $P$ there exists a $P$-saturating family $\cF$ such that some atom of $\cA(\cF)$ is not a singleton, i.e., some two elements are not separated by $\cF$, then $\sats(n,P)\le |\cF|$, and so $\sats(n,P)=O(1)$.
\end{corollary}
\begin{proof}
	We can make the non-singleton atom of $\cA(\cF)$ arbitrarily large to obtain some $\cF'\lesssim_1 \cF$ over $[n]$ that has the same size as $\cF$. By Proposition \ref{O1G}, $\cF'$ is $P$-saturating. 
\end{proof}

The contrapositive says that if $\sats(n,P)\ne O(1)$, then all atoms are singletons in any $P$-saturating family, i.e., it is separating, but then its size is at least $\log_2 n$, which proves Theorem \ref{dich}. This last thought was used in Theorem 8 of \cite{F7} to obtain lower bounds on $\sats(n,P)$ with $P$ satisfying a property that we discuss in the next subsection.

\subsection{UCTP posets with top chain}\label{sec:uctp}

In \cite{F7}, a poset property called the \emph{unique cover twin property} (UCTP) was defined. In a poset $y$ covers $x$ if there is no $z$ with $x<z<y$.
A poset $P$ is said to have UCTP if whenever $y$ covers $x$, then there is a $z$ that is comparable with one of $x$ and $y$ and is incomparable to the other one. That is either $x$ is covered by not only $y$ and thus the covering of $x$ by $y$ is not `unique', or $x$ is not the only one covered by $x$ and thus $x$ has a `twin' covered by $y$.
They have shown that for any poset $P$ with UCTP, $\sats(n,P)$ is unbounded.
We extend their theorem for a slightly more general class of posets.

A poset is called \emph{UCTP with top chain} if it consists of two parts: a poset $P_0$ that has UCTP and a chain such that every element of $P_0$ is smaller than every element of the chain.
For technical reasons, we also require $|P_0|\ge 2$ (i.e., the poset itself is not a chain).
For example, the poset on four elements defined by $a<c;b<c;c<d$ (an upside-down `Y') is a UCTP with top chain for which it was not known before whether it has an unbounded induced saturation function. 

\begin{theorem}\label{thm:uctp+chain}
	Let $P$ be a poset that has UCTP with top chain. Then any $P$-saturating family is separating, thus $\sats(n,P)\ge \log_2 n$.
\end{theorem}

\begin{proof}
	If the UCTP part of the poset does not have a largest element, then imagine that the smallest element of the chain belongs to it (this preserves the UCTP).
	From now on denote the UCTP part with $P_0$, its top element with $t$ and suppose that the top chain from $t$ has $k$ elements (including $t$), so for example for an upside-down `Y' we have $k=2$.
	Note that $P_0$ has at least two elements apart from $t$. 
	
	For a contradiction, suppose that $x$ and $y$ are not separated by the family $\cF$.
	For any set $S$, denote by $S_x=S\cup\{x\}$ and $S_{xy}=S\cup\{x,y\}$.
	Similarly, let $\cF_{xy}=\{F\in\cF\mid x,y\in F\}$, and also  $\cF_{0}=\{F\in\cF\mid x,y\notin F\}$. Note that $\cF=\cF_{0}\cup \cF_{xy}$.
	
	If $\cF_{xy}$ is $C_k$-free, then add $\{x\}$ to the family. As in $\cF$, only elements of $\cF_{xy}$ are above $x$, in a copy of $P$ no chain of length $k$ is above $x$. Thus, if we get a copy of $P$, $x$ needs to be in the top chain part, but that is impossible, since it does not have two elements under it.
	
	Otherwise, let $S_{xy}$ be a minimal set in $\cF_{xy}$ that is part of a chain on $k$ elements from $\cF_{xy}$.
	Add $S_x$ to the family. $S_x$ is in a copy of $P$ in which $S_x$ cannot be in the top chain $C_k$, as then we could remap the part of the chain starting from $S_x$ into a chain starting from $S_{xy}$. This would still be a copy of $P$ in $\cF$, contradicting that $\cF$ is $P$-free.
	Finally, if $S_x$ is in $P_0\setminus\{t\}$, then we again get a contradiction, just like in \cite{F7}. First, if $S_{xy}$ is in a copy of $P$, then, as the set $S_{xy}$ covers $S_x$, $S_x$ must contain its `twin' in $P$ (that exists due to the UCTP), so the copy of $P$ is not induced. Second, if $S_{xy}$ is not in $P$, then we could remap $S_x$ into $S_{xy}$. As they are in the same relation to all other sets of $\cF=\cF_{0}\cup \cF_{xy}$, we get a copy of $P$ in $\cF$, contradicting that $\cF$ is $P$-free.
\end{proof}

We note that if we reverse all the relations in a poset with UCTP we get another poset that has UCTP. This implies, that Theorem \ref{thm:uctp+chain} we can exchange `UCTP with top chain' by `UCTP with bottom chain', whose definition  is similar, just that instead of putting a chain above all elements, we put the chain below all elements.

\subsection{Posets with one long chain}

Let us define the \textit{generalized harp poset} $H_{P_1,P_2,\dots,P_k}$ to be the poset with a smallest element $u$ and a largest element $v$ such that $H_{P_1,P_2,\dots,P_k}\setminus \{u,v\}$ is $P_1+P_2+\dots+P_k$, the disjoint union of the posets $P_1,P_2,\dots,P_k$.

We denote by ${\cal M} _{l,k}$ the union of the middle $k$ levels\footnote{If there are two sets of middle $k$ levels then, to avoid ambiguity, let ${\cal M} _{l,k}$ denote the lower set of middle $k$ levels.} of ${\cal B}_l$ on base set $[l]$. Denote by $e^*(P)$ the maximum number $k$ such that the union of any $k$ complete and consecutive levels of any Boolean lattice is $P$-free. In particular ${\cal M} _{\ell,k}$ is $P$-free for every $\ell$. The converse is also essentially true:

\begin{observation}\label{obs:estar}
	If $e^*(P)=k$ then ${\cal M} _{\ell,k+1}$ contains $P$ for every big enough $\ell$. 
\end{observation}
\begin{proof}
	If $e^*(P)=k$ then by definition for some $\ell'$ and $a$, the levels $a,a+1,\dots,a+k-1$ of $B_{\ell'}$ contain $P$. In this case for any $\ell\ge 3\ell'$ the family $\cM_{\ell,k}$ also contains $P$ as one can consider level $a,a+1,\dots,a+k-1$ of $\{H: A\subseteq H\subseteq B\}$ with $A,B\subseteq [\ell]$, $|B|-|A|=\ell'$ (which is isomorphic as a poset to $B_{\ell'}$) such that $|A|+a$ equals the rank of the lowest level of $\cM_{\ell,k}$.
\end{proof}

\begin{theorem}\label{thm:chain+shallow}
If $P$ is a poset with $e^*(P)\le k-2$, then $\sats(n,C_k+P)\le K_P$ for some constant independent of $n$.
\end{theorem}

It follows from Observation \ref{obs} that we also have $\sats(n,H_{C_k,P})\le K_P$.



\begin{proof}
	If $P$ is empty, then the statement is true by the result of \cite{G6} for chains, so from now on we will assume that $P$ is non-empty.
	For any pair $k,\ell$ of positive integers let $\cF=\cF^0_{\ell,k-1}\cup\cF^1_{\ell,k-1}\cup \{\emptyset,[n]\}$ with
$$ \cF^0_{\ell,k-1}:=\{F: 1\notin F\in \cM_{\ell,k-1}\}, \hskip 0.2truecm \cF^1_{\ell,k-1}:=\{F\cup ([n]\setminus [\ell]): 1\in F\in \cM_{\ell,k-1}\}.$$

We choose $l$ big enough so that $l\ge 10k$ and $\cM_{\ell-2,k-1}$ contains $P$. This can be done by Observation \ref{obs:estar}.

Observe that $\cF\setminus \{\emptyset,[n]\}$ does not contain a chain of size $k$. Indeed, the mapping $f:\cF\rightarrow B_\ell$ with $f(F)=F\cap [\ell]$ is a poset-isomorphism from $\cF$ to $f(\cF)$, and $f(\cF\setminus \{\emptyset,[n]\})=\cM_{\ell,k-1}$.
Thus the only way to embed $C_k$ would be if its top or bottom element was mapped to $[n]$ or $\emptyset$, but this contradicts that there are no relations between $C_k$ and $P$.
Therefore $\cF$ is induced $(C_k+P)$-free.

Next we show that for any $G\in B_n\setminus \cF$, the family $\cF\cup \{G\}$ contains an induced copy of $C_k+P$. We consider several cases. In all cases, we will find a $k$-chain in $\cF \cup \{G\}$ such that for some element $x$ of the smallest set of the chain and an element $y$ that does not belong to the largest set of the chain, the family $\cF^{\overline{x},y}=\{F\in\cF:x\notin F,y\in F\}$ is isomorphic as a poset to $\cM_{\ell-2,k-1}$. If so, then by virtue of $x$ and $y$ the chain and $\cF^{\overline{x},y}$ are incomparable and $\cF^{\overline{x},y}$ being isomorphic to $\cM_{\ell-2,k-1}$ implies that $\cF^{\overline{x},y}$ contains a copy of $P$. 

\bigskip

\textsc{Case I.} $G\cap [\ell] \in \cM_{\ell,k-1}$.

Then $10k\le \ell$ implies $4k\le |G\cap [\ell]|\le \ell-4k$, so we can fix $x$ and $y$ with $x,y\in\{2,3,\dots,l\}$, $x\in G\cap [\ell]$, $y\in [\ell]\setminus G$. Let  $M_1 \subsetneq M_2 \subsetneq \dots \subsetneq M_{k-1}$ be a chain of length $k-1$ in $\cM_{\ell,k-1}$ such that $G\cap [\ell]=M_i$ for some $i$, $x\in M_1,y\notin M_{k-1}$ and $1\notin M_{i-1}, 1\in M_{i+1}$. Such a chain exists as $4k\le |G\cap [\ell]|\le \ell-4k$ holds. If $1\notin G$, then $M_1, M_2, \dots,M_{i-1},G \cap [\ell], G, M_{i+1}\cup ([n]\setminus [\ell]), \dots, M_{k-1}\cup ([n]\setminus [\ell])$ is a $k$-chain in $\cF \cup \{G\}$ and $\cF^{\overline{x},y}$ is as desired. While if $1\in G$, then $M_1,\dots,M_{i-1},G, G \cup ([n]\setminus [\ell]), M_{i+1}\cup ([n]\setminus [\ell]), \dots, M_{k-1}\cup ([n]\setminus [\ell])$ is a $k$-chain in $\cF \cup \{G\}$ and $\cF^{\overline{x},y}$ is as desired.

\bigskip

\textsc{Case II.} $G\cap [\ell]\notin \cM_{\ell,k-1}$.

Now there are lots of chains $M_1\subsetneq M_2 \subsetneq \dots \subsetneq M_{k-1}$ in $\cM_{\ell,k-1}$ that are extendable with $G\cap [\ell]$. If we can pick $x,y\in \{2,3,\dots,\ell\}$ with $x$ belonging to all these sets and $y$ belonging to none of these sets, then we can proceed as in Case I. The only cases when we cannot pick $x$ and $y$ are $G\cap [\ell]\in\{\emptyset,\{1\},\{2,3,\dots,\ell\},[\ell]\}$.
\begin{enumerate}
	\item 
	If $G\cap [\ell]=\emptyset$ or $G\cap [\ell]=\{1\}$, then we pick $y\in\{2,3,\dots,\ell\}$ and consider a chain $M_1\subsetneq M_2 \subsetneq \dots \subsetneq M_{k-1}$ in $\cM_{\ell,k-1}$ with $1\in M_1,y\notin M_{k-1}$. Then the $k$-chain $G, M_{1}\cup([n]\setminus [\ell]), \dots M_{k-1}\cup([n]\setminus [\ell])$ is incomparable to $\{F\in \cF:1\notin F,y\in F\}$ (here we use $G\neq \emptyset$), which is isomorphic to $\cM_{\ell-2,k-1}$.
	\item 
	If $G\cap [\ell]=[\ell]$ or $G\cap [\ell]=\{2,3,\dots,\ell\}$, then we pick $x\in\{2,3,\dots,\ell\}$ and consider a chain $M_1\subsetneq M_2 \subsetneq \dots \subsetneq M_{k-1}$ in $\cM_{\ell,k-1}$ with $x\in M_1,1\notin M_{k-1}$. Then the $k$-chain $M_{1}, \dots M_{k-1},G$ is incomparable to $\{F\in \cF:1\in F,x\notin F\}$ (here we use $G\neq [n]$), which is isomorphic to $\cM_{\ell-2,k-1}$.\qedhere
\end{enumerate}
\end{proof}

\subsection{The poset $2C_2$}

In this subsection, we prove that for the poset $2C_2$ of two incomparable pairs, the induced saturation number is unbounded. More precisely, we obtain $n+2\le\sats(n,2C_2)\le 2n$. The upper bound is attained by the $2C_2$-saturating family consisting of all singletons and any maximal chain.
\vskip 0.3truecm

First we prove the upper bound. 

\begin{proposition}\label{2C2UB}
	For any integer $n\ge 3$, $\sats(n,2C_2)\le 2n$.
\end{proposition}

\begin{proof}
	The family $\cF$ consists of a full chain and the singletons. Without loss of generality, we may choose $\left\{\emptyset, \{1\}, \{1,2\}, \{1,2,3\}, \ldots, [n], \{2\}, \{3\}, \ldots, \{n\}\right\}$. It is clear that $\cF$ has no induced copy of $2C_2$. Now consider a set $S\not\in\cF$. Let $m$ be the maximum element in $S$ and let $\ell$ be the least element not in $S$. Because $S\not\in\cF$, we know that $\ell < m$. In addition, $S$ must contain an element not in $\{1,2\}$ so $2 < m$.
	
	Consider $\cF \cup \{S\}$. If $\ell \in \{1,2\}$, then an induced copy of $2C_2$ is $\left\{\{\ell\}, \{1,2\}, \{m\}, S\right\}$. If $\ell \ge 3$, then an induced copy of $2C_2$ is $\left\{\{\ell\}, \{1,2,\ldots,\ell\}, \{m\}, S\right\}$. Thus, $\cF$ is saturating induced $2C_2$-free.
\end{proof}

Let us now turn to the lower bound. For a family $\cF$ and a set $G$ let us define $\cD_{\cF}(G)=\{F\in \cF: F\subsetneq G\}$.

\begin{proposition}\label{order}
If $\cF$ is induced $2C_2$-free, then for any $F,F'\in \cF$ one of the following three possibilities hold.
\begin{itemize}
\item
$\cD_{\cF}(F)\subsetneq \cD_{\cF}(F')$,
\item
$\cD_{\cF}(F')\subsetneq \cD_{\cF}(F)$,
\item
$\cD_{\cF}(F)= \cD_{\cF}(F')$.
\end{itemize}
\end{proposition}

\begin{proof}
If $G\in \cD_{\cF}(F)\setminus \cD_{\cF}(F')$ and $G'\in \cD_{\cF}(F') \setminus \cD_{\cF}(F)$, then $F,F',G,G'$ form a copy of $2C_2$.
\end{proof}

\begin{theorem}\label{2+2}
If $\cF\subseteq 2^{[n]}$ is saturating induced $2C_2$-free, then $\cF$ contains a maximal chain in $[n]$. In particular, $\sats(n,2C_2)\ge n+2$ holds.
\end{theorem}

\begin{proof}
Let $\cF\subseteq 2^{[n]}$ be a saturating induced $2C_2$-free family of sets. Clearly, $[n]$ and $\emptyset$ both belong to $\cF$ as they are comparable to every other set in $2^{[n]}$. For two sets $F,F'\in \cF$, we define the relation $F<F'$ if $\cD_{\cF}(F)\subsetneq \cD_{\cF}(F')$ holds. By Proposition \ref{order} we obtain that we either have $F<F'$ or $F'<F$ or $\cD_{\cF}(F)= \cD_{\cF}(F')$. Clearly, $<$ is transitive, thus we can enumerate the sets of $\cF$ as $[n]=F_1, F_2, \dots,F_m=\emptyset$ such that $i<j$ implies $F_i>F_j$ or $\cD_{\cF}(F_i)= \cD_{\cF}(F_j)$. For any $j=1,2,\dots,m$ let $G_j=\cap_{i=1}^jF_i$. In particular, we have $G_1=F_1=[n]$ and $G_m=F_m=\emptyset$ and the $G_j$'s form a chain.

\begin{claim}\label{claim}
For any $h=1,2,\dots, m$ we have $\cD_{\cF}(G_h)\subseteq \cD_{\cF}(F_h)\subseteq \cD_{\cF}(G_h)\cup \{G_h\}$.
\end{claim}

\begin{proof}[Proof of Claim]
By definition, we have $G_h=\cap_{i=1}^hF_i\subseteq F_h$ and this clearly implies $\cD_{\cF}(G_h)\subseteq \cD_{\cF}(F_h)$. Also, the way we enumerated the $F_i$'s implies $\cD_\cF(F_h)\subseteq \cD_\cF(F_i)$ for all $1\le i<h$, so $\cD_\cF(F_h)=\cap_{i=1}^h\cD_\cF(F_i)\subseteq \cD_\cF(G_h)\cup \{G_h\}$.
\end{proof}

 We show that if $G_{j+1}\subseteq X \subsetneq G_j$, then $X$ must belong to $\cF$. Suppose not, then adding $X$ to $\cF$ creates an induced copy of $2C_2$ and thus there must exist a pair $A\subseteq B$ in $\cF$ incomparable to $X$. Clearly, $A<B$, so $A=F_k$, $B=F_\ell$ for some $\ell<k$. If $\ell\le j$, then $X\subsetneq G_j\subseteq F_\ell=B$ gives a contradiction. Finally suppose $\ell \ge j+1$. Applying Claim \ref{claim} to $h=j+1$ shows that $$A\in \cD_{\cF}(F_\ell)\subseteq \cD_{\cF}(F_{j+1})\subseteq \cD_{\cF}(G_{j+1})\cup \{G_{j+1}\}\subseteq \cD(X) \cup \{X\},$$
which contradicts the assumption that $A$ and $X$ are incomparable.

This completes the proof of the fact that a maximal chain is contained in $\cF$. As $2C_2$ consists of two incomparable pairs, and in a chain all pairs of sets are comparable, by the saturation property, $\cF$ must contain at least one set not in the maximal chain. Therefore $|\cF|\ge n+2$ holds.
\end{proof}

It is a natural question whether the lower bound can be improved to $2n$ by proving that any $2C_2$-saturating family contains an antichain of size $n$.
We could neither prove, nor disprove this. 

\subsubsection*{The posets k$C_2$ for $k\ge 3$}

Define the (circular) \emph{interval lattice} $\cI_n$ as the collection of subsets of $[n]$ that are of the form $\{i,\ldots,j\}$, or their complement is of this form.
Denote by $\hat\cI_k$ the collection of subsets of $[n]$ that we get from $\cI_k$ by replacing every occurrence of $\{k\}$ with $\{k,\ldots,n\}$.
The families $\hat\cI_k$ are natural candidates for saturating $C_2+\ldots+C_2$, more precisely, the largest number of induced copies of $C_2$'s that they contain is $\lfloor \frac{2k}3\rfloor C_2$.
We have verified these claims by computer and also by hand, but our arguments are not particularly interesting, just mainly a case analysis, so we do not include them here.

\begin{proposition}\label{prop:3C2}
	$\hat\cI_4$ is saturating for $3C_2$, thus $\sats(n,3C_2)\le 14$.
\end{proposition}

\begin{proposition}\label{prop:5C2}
	$\hat\cI_7$ is saturating for $5C_2$, thus $\sats(n,5C_2)\le 44$.
\end{proposition}

\begin{proposition}
	$\hat\cI_{10}$ is \emph{not} saturating for $7C_2$.
\end{proposition}
\begin{proof}
	Let $n=10$ and add the set $\{1,3,5,7,9\}$.
\end{proof}


Based on this, it is not clear what to conjecture, but with a computer we have found a saturating family also for $7C_2$ (see later).


\subsection{Greedy colex process for induced saturation and butterfly}

In this subsection we consider the induced version of Algorithm \ref{algo} that showed $\sat(n,P)\le 2^{|P|-2}$ for any poset $P$. As a reminder, we build a family $\cF\subseteq 2^{[n]}$ as follows. We enumerate the sets of $2^{[n-1]}$ in colex order: $F_1,F_2,\dots, F_{2^{n-1}}$. Then setting $\cF_0=\emptyset$, we repeat the following: once $\cF_{i-1}$ is defined, we add $F_i$ if it does not create a copy of $P$, then we add $G_i=[n]\setminus F_i$ if it does not create a copy of $P$. This gives $\cF_i$. Clearly, $\cF=\cF_{2^{n-1}}$ is saturating $P$-free.

The greedy colex process gave non-linear bounds on some posets, but performed well on others. In fact, for the butterfly poset, $\bowtie$, the previous best upper bound on $\sats(n,\bowtie)$ due to Ferrara et al.~\cite{F7} was quadratic in $n$, but we have managed to improve it to linear by analyzing the output of the greedy colex process. We need to define the resulting family. Let $\cT_1=\{\emptyset\}$, $\cT_{2}=\{\{1\},\{2\},\{1,2\}\}$, $\cT_3=\{\{3\},\{1,3\},\{2,3\}\}$. For $k\ge 2$ let $$\cT_{2k}=\{\{1,4,6,\dots,2k\},\{2,4,6,\dots,2k\},\{1,2,4,6,\dots,2k\}\}$$ and $$\cT_{2k+1}=\{\{3,5,\dots,2k+1\},\{1,3,5,\dots,2k+1\},\{2,3,5,\dots,2k+1\}\}.$$ For any $1\le j<n$, let $\cT_{j,n}=\{[n]\setminus T: T\in \cT_j\}$. Finally, let $\cH_n=\cup_{j=1}^{n-1}(\cT_j\cup \cT_{j,n})$.

\begin{theorem}\label{butterfly}
For $n\ge 3$ the greedy colex induced $\bowtie$-free process produces $\cH_n$, in particular $\sats(n,\bowtie)\le 6n-10$.
\end{theorem}

\begin{proof}
The cases $n=3,4$ can be verified by hand, then we apply induction on $n$ and assume that the statement of the theorem holds for $n$. Observe that the mapping with $F_i \mapsto F_i$, $[n]\setminus F_i \mapsto [n+1] \setminus F_i$ is inclusion and non-inclusion preserving from $2^{[n]}$ to the first half of the greedy colex order of $2^{[n+1]}$. This shows that when we run the greedy colex process on $2^{[n+1]}$ (for the sets that do not contain $n,n+1$ and their complements), we obtain $\cF_{2^{n-1}}=f[\cH_n]=\cup_{i=1}^{n-1}(\cT_i \cup \cT_{i,n+1})$. So we need to show that in the second half of the process exactly sets of $\cT_n \cup \cT_{n,n+1}$ are added. Observe that in the second half of the process one considers the sets that contain exactly one of $n$ and $n+1$.

We deal with the sets according to their intersection with $\{1,2,3\}$. If $F\cap \{1,2,3\}=\emptyset$, then $F,\{3\},[n+1]\setminus \{1\}, [n+1]\setminus \{2\}$ form a butterfly, so $F$ cannot be added. Similarly, if $\{1,2,3\}\subseteq F$, then $F,[n+1]\setminus \{3\}, \{1\},\{2\}$ form a butterfly, so $F$ cannot be added in this case, either.

For sets with $|F\cap \{1,2,3\}|=1$ or $2$, let us consider the following six chains.
\begin{itemize}
\item
Let $\cC_1=\{\{1\} \subset \{1,4\} \subset \{1,4,6\} \subset ...\subset \{1,4,6,...,2m-2\}\subset \{1,4,6,...,2m\} = [n+1]\setminus \{2,3,5,...,2m-1\} \subset [n+1]\setminus \{2,3,5,...,2m-3\} \subset ... \subset [n+1] \setminus \{2,3\}\}$ if $n+1=2m$ is even and let $\cC_1=\{\{1\} \subset \{1,4\} \subset \{1,4,6\} \subset ...\subset \{1,4,6,...,2m\}\subset \{1,4,6,...,2m,2m+1\} = [n+1]\setminus \{2,3,5,...,2m-1\} \subset [n+1]\setminus \{2,3,5,...,2m-3\} \subset ... \subset [n+1] \setminus \{2,3\}\}$ if $n+1=2m+1$ is odd.
\item
Let $\cC_2=\{\{2\} \subset \{2,4\} \subset \{2,4,6\} \subset ...\subset \{2,4,6,...,2m-2\}\subset \{2,4,6,...,2m\} = [n+1]\setminus \{1,3,5,...,2m-1\} \subset [n+1]\setminus \{1,3,5,...,2m-3\} \subset ... \subset [n+1] \setminus \{1,3\}\}$ if $n+1=2m$ is even and let $\cC_2=\{\{2\} \subset \{2,4\} \subset \{2,4,6\} \subset ...\subset \{2,4,6,...,2m\}\subset \{2,4,6,...,2m,2m+1\} = [n+1]\setminus \{1,3,5,...,2m-1\} \subset [n+1]\setminus \{1,3,5,...,2m-3\} \subset ... \subset [n+1] \setminus \{1,3\}\}$ if $n+1=2m+1$ is odd.
\item
Let $\cC_3=\{\{3\} \subset \{3,5\} \subset \{3,5,7\} \subset ...\subset \{3,5,7,...,2m-1\}\subset \{3,5,7,...,2m-1,2m\} = [n+1]\setminus \{1,2,4,...,2m-2\} \subset [n+1]\setminus \{1,2,4,...,2m-4\} \subset ... \subset [n+1] \setminus \{1,2\}\}$ if $n+1=2m$ is even and let $\cC_3=\{\{3\} \subset \{3,5\} \subset \{3,5,7\} \subset ...\subset \{3,5,7,...,2m+1\}\subset \{3,5,7,...2m-1,2m,2m+1\} = [n+1]\setminus \{1,2,4,...,2m-2\} \subset [n+1]\setminus \{1,2,4,...,2m-4\} \subset ... \subset [n+1] \setminus \{1,2\}\}$ if $n+1=2m+1$ is odd.
\item
Let $\cC_{1,2}=\{\{1,2\} \subset \{1,2,4\} \subset \{1,2,4,6\} \subset ...\subset \{1,2,4,6,...,2m-2\}\subset \{1,2,4,6,...,$ $2m\} = [n+1]\setminus \{3,5,...,2m-1\} \subset [n+1]\setminus \{3,5,...,2m-3\} \subset ... \subset [n+1] \setminus \{3\}\}$ if $n+1=2m$ is even and let $\cC_{1,2}=\{\{1,2\} \subset \{1,2,4\} \subset \{1,2,4,6\} \subset ...\subset \{1,2,4,6,...,2m\}\subset \{1,2,4,6,...,2m,2m+1\} = [n+1]\setminus \{3,5,...,2m-1\} \subset [n+1]\setminus \{3,5,...,2m-3\} \subset ... \subset [n+1] \setminus \{3\}\}$ if $n+1=2m+1$ is odd.
\item
Let $\cC_{1,3}=\{\{1,3\} \subset \{1,3,5\} \subset \{1,3,5,7\} \subset ...\subset \{1,3,5,7,...,2m-1\}\subset \{1,3,5,7,...,$ $2m-1,2m\} = [n+1]\setminus \{2,4,...,2m-2\} \subset [n+1]\setminus \{2,4,...,2m-4\} \subset ... \subset [n+1] \setminus \{2\}\}$ if $n+1=2m$ is even and let $\cC_{1,3}=\{\{1,3\} \subset \{1,3,5\} \subset \{1,3,5,7\} \subset ...\subset \{1,3,5,7,...,2m+1\}\subset \{1,3,5,7,...2m-1,2m,2m+1\} = [n+1]\setminus \{2,4,...,2m-2\} \subset [n+1]\setminus \{2,4,...,2m-4\} \subset ... \subset [n+1] \setminus \{2\}\}$ if $n+1=2m+1$ is odd.
\item
Let $\cC_{2,3}=\{\{2,3\} \subset \{2,3,5\} \subset \{2,3,5,7\} \subset ...\subset \{2,3,5,7,...,2m-1\}\subset \{2,3,5,7,...,$ $2m-1,2m\} = [n+1]\setminus \{1,4,...,2m-2\} \subset [n+1]\setminus \{1,4,...,2m-4\} \subset ... \subset [n+1] \setminus \{1\}\}$ if $n+1=2m$ is even and let $\cC_{2,3}=\{\{2,3\} \subset \{2,3,5\} \subset \{2,3,5,7\} \subset ...\subset \{2,3,5,7,...,2m+1\}\subset \{2,3,5,7,...2m-1,2m,2m+1\} = [n+1]\setminus \{1,4,...,2m-2\} \subset [n+1]\setminus \{1,4,...,2m-4\} \subset ... \subset [n+1] \setminus \{1\}\}$ if $n+1=2m+1$ is odd.
\end{itemize}
Now observe that $|\cC_S\cap (\cT_j\cup \cT_{j,n+1})|=1$ for any $S\subseteq \{1,2,3\}$, $|S|=1,2$ and $2\le j\le n$, i.e., these chains partition $\cH_{n+1}\setminus \{\emptyset,[n+1]\}$ and all of them contain exactly one set that will be added in the second half of the greedy colex process. So if $F\cap \{1,2,3\}=S$ with $|S|=1$, then $[n+1]\setminus S',[n+1]\setminus S'',F,G$ form a butterfly where $S'$ and $S''$ are the other two singleton subsets of $\{1,2,3\}$, and $G$ is any member of $\cC_S$ that is incomparable to $F$. As the sizes of consecutive sets of these chains differ by 1, there are only two sets $F$ with $F\cap \{1,2,3\}=S$ that are comparable to all sets of $\cC_S$ that are in $\cF_{2^{n-1}}$: the unique set $F_S\in \cC_S \cap (\cT_n \cup \cT_{n,n+1})$ and the other set of the same size between the two sets of $\cC_S$ neighboring $F_S$. It is easy to verify that in the greedy colex process $F_S$ comes first, so it will be added, and the other will not as together with $F_S$ and $[n+1]\setminus S',[n+1]\setminus S''$ it would form a butterfly.

An analogous argument is valid for the case  $F\cap \{1,2,3\}=S$ with $|S|=2$, and sets $\{s_1\},\{s_2\},F,G$, where $s_1$ and $s_2$ are the two elements of $S$ and $G$ is any member of $\cC_S$ that is incomparable to $F$.
\end{proof}

As Ivan has recently obtained the lower bound $\sats(n,\bowtie)\ge n+1$ \cite{I}, we get the following corollary.

\begin{corollary}
	For the butterfly poset, we have $\sats(n,\bowtie)=\Theta(n)$.
\end{corollary}

\subsection{Experimental results}

Here we report further upper bounds that were found by running the greedy colex process. Observe that if the algorithm for poset $P$ and ground set $[m]$ returns a family that does not separate $m-1$ and $m$, then Corollary \ref{cor:const} yields that the size of the resulting family is an upper bound on $sat^*(n,P)$ for any $n\ge m$. When running our algorithm, we can also exploit the fact that if some set $F$ whose largest element is $m$ forms a copy of $P$ with some collection of sets from $\cF_{2^{m-1}}$, then $F\setminus\{m\}\cup \{m'\}$ will also form a copy of $P$ with the same collection of sets for any $m<m'<n$; this significantly reduces the sets we have to test, sometimes even to linear in $n$. When checking $P$-freeness of a family $\cF$ for a $P$ that is a disjoint union of chains, then it is sufficient to search for a copy of $P$ in the Hasse-diagram of $\cF$, as any other copy could be transformed into one such (similarly as is done at the end of the proof of Theorem \ref{thm:uctp+chain}); this enabled us to run our code for large posets that are the disjoint union of chains by maintaining the Hasse-diagram of $\cF_i$.

In the next proposition we list some posets that we have found interesting.
The posets $C_2+2C_1$ and $C_3+C_1$ are there to have some concrete upper bound for all posets with at most $4$ elements.
The other posets are disjoint unions of chains that do not satisfy the conditions of Theorem \ref{thm:chain+shallow}, so we had no upper bound for them, except for $3C_2$ and $5C_2$, covered by Propositions \ref{prop:3C2} and \ref{prop:5C2}, respectively.
From among the posets on $5$ elements that are covered by neither Theorem \ref{thm:uctp+chain} nor Theorem \ref{thm:chain+shallow}, the greedy colex process could find a constant upper bound only for some that are the disjoint union of smaller posets, such as $2C_2+C_1$.
We obtained all the following bounds via computer and give a full description of the families that witness these bounds in the Appendix.

\begin{proposition}\label{prop:const}
	\begin{equation*}
	\begin{split}
	\\
	\sats(n,3C_2)&\le 14,\\
	\sats(n,5C_2)&\le 42,\\
	\sats(n,7C_2)&\le 60,\\
	\sats(n,3C_3)&\le 28,\\
	\sats(n,3C_4)&\le 52,
	\end{split}
	\quad~~\quad
	\begin{split}
	\sats(n,C_2+2C_1)&\le 8,\\
	\sats(n,C_3+C_1)&\le 8,\\
	\sats(n,2C_2+C_1)&\le 12,\\
	\sats(n,2C_3+C_1)&\le 28,\\
	\sats(n,2C_3+C_2)&\le 20,\\	
	\sats(n,2C_3+2C_1)&\le 26,
	\end{split}
	\quad~~\quad
	\begin{split}
	\\
	\sats(n,2C_4+C_1)&\le 60,\\
	\sats(n,2C_4+2C_1)&\le 68,\\
	\sats(n,2C_4+C_2)&\le 54,\\
	\sats(n,2C_4+C_3)&\le 38,\\
	\sats(n,2C_4+2C_2)&\le 46.
	\end{split}
	\end{equation*}

	
%
%
%
%
%
%
%
%
%
%
%
%
%
%
%
\end{proposition}

We also had many posets for which numerical evidence suggested certain upper bounds, that could be converted to theorems, just like in the case of Theorem \ref{butterfly}.
Most of the data supporting these bounds can be found among the source files of this paper on arXiv \cite{arxiv}.
These bounds might not always be the right magnitude, e.g., for $2C_3$ the greedy colex process gives a quadratic bound but we can prove a linear upper bound.

\begin{proposition}
	For any positive integer $n$ we have $\sats(n,2C_3)\le 3n-1$.
\end{proposition}

\begin{proof}
	For $n\le 3$, this is true because $2^n\le 3n-1$.
	For $n\ge 3$, consider a family that contains a full chain, all singletons and all co-singletons, e.g.,
	\[
	\cF_n:=\{\emptyset\} \cup \{[j]:1\le j\le n\} \cup  \{ \{j\}:1\le j\le n\} \cup  \{[n]\setminus \{j\}:1\le j\le n\}.
	\]
	We claim that $\cF_n\subseteq 2^{[n]}$ is $2C_3$-saturating. Indeed, if $G\notin \cF_n$, then for $s=|G|$ we have  $2\le s\le n-2$. Then there exists $x,y\in [n]$ with $x\in [s]\setminus G$, $y\in G\setminus [s]$. Then $\{x\},[s],[n]\setminus \{y\}$ and $\{y\},G,[n]\setminus \{x\}$ are two incomparable 3-chains in $\cF \cup \{G\}$.
\end{proof}

\smallskip
We do not know how $\sats(n,2C_k)$, behaves (except for $k=2$, see Theorem \ref{2+2}); we have neither a non-constant lower bound, nor a linear upper bound.
The greedy colex process gives a cubic bound for $2C_4$.
Even worse, for some posets it gives just an exponential bound (for $n\le 10$); one such example is $\Diamond'$, the poset on $5$ elements given by $A<B<C$ and $A<B'<C,C'$, obtainable from the diamond poset, $\Diamond$, with the addition of one more element. 
We can, however, also prove a linear upper bound for $\Diamond'$.

\begin{proposition}
	For any positive integer, we have $\sats(n,\Diamond')\le 2n$.
\end{proposition}

\begin{proof}
	For $n\le 2$, this is true because $2^n\le 2n$.
	For $n\ge 3$, consider a family that contains a full chain, and all singletons, e.g.,
	\[
	\cF_n:=\{\emptyset\} \cup \{[j]:1\le j\le n\} \cup  \{ \{j\}:1\le j\le n\}.
	\]
	
	We claim that $\cF_n$ is $\Diamond'$-saturating.
	Indeed, if $G \notin \cF_n$, then it can be written as $G=[i] \cup G'$ with $i+1 \notin G'\ne \emptyset$. Then $\emptyset, \{1\},\{i+1\},[i+1],G$ form a copy of $\Diamond'$ provided $1\le i$.
	In case $i=0$ and thus $1 \notin G$, then $G=G'$ and $|G'|\ge 2$, so for the smallest element of $G$ we have $1<m$ and $\emptyset,\{1\},\{m\},[m],G$ form a copy of $\Diamond'$.
\end{proof}

Another interesting case seems to be $\sats(n,2C_k+2C_{k-1})$, for which the greedy colex gave a linear upper bound for $k\le 4$; is it always unbounded?
It might be easier to show that $\sats(n,2C_k+C_1)$ is bounded for every $k\ge 2$.

\subsection*{Concluding remarks}

\begin{question}
	Is it decidable for a poset $P$ whether $\sats(n,P)$ is bounded or not?
\end{question}

This problem is obviously recursively enumerable, but we could find no witness for unboundedness.
Note that the size of the witness for boundedness can be exponential in $|P|$.
Could it be even larger?

Another question is whether the greedy colex process can always verify boundedness if it runs long enough.

\begin{question}
	If $\sats(n,P)$ is bounded, does the greedy colex process find a bounded family? If yes, after which $n$ will $\cF_{2^{n-1}}$ remain unchanged?
\end{question}

We also do not know how fast $\sats(n,P)$ can grow as a function of $n$.

\begin{question}
	Is $\sats(n,P)=O(n)$ for every poset $P$?
\end{question}

\subsubsection*{Acknowledgements}

We would like to thank our anonymous referees for the careful reading of the paper and for spotting several inaccuracies in the first version of this paper.

The research of Keszegh and P\'alv\"olgyi was supported by the Lend\"ulet program of the Hungarian Academy of Sciences (MTA), under grant number LP2017-19/2017.
The research of Keszegh was also supported by the National Research, Development and Innovation Office -- NKFIH under the grant K 132696. Martin's research was partially supported by Simons Foundation Collaboration Grant  \#353292 and by the J. William Fulbright Educational Exchange Program. The research of Patk\'os was supported partially by the grant of Russian Government N 075-15-2019-1926 and by the National Research, Development and Innovation Office - NKFIH under the grants FK 132060 and SNN 129364.

\appendix


\section{Output of the colex algorithm}\label{app:outputs}

Here we give the output of the colex algorithm for various posets, verifying the upper bounds of Proposition \ref{prop:const}.
Note that by Proposition \ref{O1G}, it is enough to find a non-separating family for any $n$; in our examples always the two largest elements are not separated.

\subsection{Chains of size at most $2$}

\noindent $\sats(n,C_2+2C_1)\le 8$.
\begin{spverbatim}
	A saturating family of size 8 is: [{}, {1, 2, 3, 4}, {1}, {2, 3, 4}, {2}, {1, 3, 4}, {1, 2}, {3, 4}]
\end{spverbatim}

\bigskip\noindent $\sats(n,2C_2+C_1)\le 12$.
\begin{spverbatim}
	A saturating family of size 12 is: [{}, {1, 2, 3, 4, 5}, {1}, {2, 3, 4, 5}, {2}, {1, 3, 4, 5}, {1, 2}, {3, 4, 5}, {3}, {1, 2, 4, 5}, {1, 2, 3}, {4, 5}]
\end{spverbatim}

\bigskip\noindent $\sats(n,3C_2)\le 14$.
\begin{spverbatim}
	A saturating family of size 14 is: [{}, {1, 2, 3, 4, 5}, {1}, {2, 3, 4, 5}, {2}, {1, 3, 4, 5}, {1, 2}, {3, 4, 5}, {3}, {1, 2, 4, 5}, {1, 3}, {2, 4, 5}, {1, 2, 3}, {4, 5}]
\end{spverbatim}

\bigskip\noindent $\sats(n,5C_2)\le 42$.
\begin{spverbatim}
	A saturating family of size 42 is: [{}, {1, 2, 3, 4, 5, 6, 7, 8}, {1}, {2, 3, 4, 5, 6, 7, 8}, {2}, {1, 3, 4, 5, 6, 7, 8}, {1, 2}, {3, 4, 5, 6, 7, 8}, {3}, {1, 2, 4, 5, 6, 7, 8}, {1, 3}, {2, 4, 5, 6, 7, 8}, {2, 3}, {1, 4, 5, 6, 7, 8}, {1, 2, 3}, {8, 4, 5, 6, 7}, {4}, {1, 2, 3, 5, 6, 7, 8}, {1, 4}, {2, 3, 5, 6, 7, 8}, {2, 4}, {1, 3, 5, 6, 7, 8}, {1, 2, 4}, {8, 3, 5, 6, 7}, {1, 3, 4}, {8, 2, 5, 6, 7}, {1, 2, 3, 4}, {8, 5, 6, 7}, {5}, {1, 2, 3, 4, 6, 7, 8}, {3, 5}, {1, 2, 4, 6, 7, 8}, {1, 2, 3, 5}, {8, 4, 6, 7}, {4, 5}, {1, 2, 3, 6, 7, 8}, {1, 2, 3, 4, 5}, {8, 6, 7}, {6}, {1, 2, 3, 4, 5, 7, 8}, {4, 6}, {1, 2, 3, 5, 7, 8}]
\end{spverbatim}

\bigskip\noindent $\sats(n,7C_2)\le 60$.
\begin{spverbatim}
	A saturating family of size 60 is: [{}, {1, 2, 3, 4, 5, 6, 7, 8, 9}, {1}, {2, 3, 4, 5, 6, 7, 8, 9}, {2}, {1, 3, 4, 5, 6, 7, 8, 9}, {1, 2}, {3, 4, 5, 6, 7, 8, 9}, {3}, {1, 2, 4, 5, 6, 7, 8, 9}, {1, 3}, {2, 4, 5, 6, 7, 8, 9}, {2, 3}, {1, 4, 5, 6, 7, 8, 9}, {1, 2, 3}, {4, 5, 6, 7, 8, 9}, {4}, {1, 2, 3, 5, 6, 7, 8, 9}, {1, 4}, {2, 3, 5, 6, 7, 8, 9}, {2, 4}, {1, 3, 5, 6, 7, 8, 9}, {1, 2, 4}, {3, 5, 6, 7, 8, 9}, {3, 4}, {1, 2, 5, 6, 7, 8, 9}, {1, 3, 4}, {2, 5, 6, 7, 8, 9}, {2, 3, 4}, {1, 5, 6, 7, 8, 9}, {1, 2, 3, 4}, {8, 9, 5, 6, 7}, {5}, {1, 2, 3, 4, 6, 7, 8, 9}, {1, 5}, {2, 3, 4, 6, 7, 8, 9}, {2, 5}, {1, 3, 4, 6, 7, 8, 9}, {1, 2, 5}, {3, 4, 6, 7, 8, 9}, {3, 5}, {1, 2, 4, 6, 7, 8, 9}, {1, 3, 5}, {2, 4, 6, 7, 8, 9}, {2, 4, 5}, {1, 3, 6, 7, 8, 9}, {1, 2, 3, 4, 5}, {8, 9, 6, 7}, {6}, {1, 2, 3, 4, 5, 7, 8, 9}, {1, 2, 4, 6}, {8, 9, 3, 5, 7}, {5, 6}, {1, 2, 3, 4, 7, 8, 9}, {1, 2, 3, 4, 5, 6}, {8, 9, 7}, {1, 2, 3, 4, 5, 7}, {8, 9, 6}, {6, 7}, {1, 2, 3, 4, 5, 8, 9}]
\end{spverbatim}

\begin{rem}	
	All of the above families can be described as follows:
	\begin{itemize}
		\item $C_2+2C_1$: $\emptyset$, 1, 2, 12 and their complements
		
		\item $2C_2+C_1$: $\emptyset$, 1, 2, 12, 3, 123 and their complements
		
		\item $3C_2$: $\emptyset$, 1, 2, 3, 12, 13, 123 and their complements
		
		\item $5C_2$: $\emptyset$, 1, 2, 3, 4, 5, 6, 12, 13, 23, 14, 24, 35, 45, 46, 123, 124, 134, 1234, 1235, 12345 and their complements
		
		\item $7C_2$: $\emptyset$, 1, 2, 3, 4, 5, 6, 12, 13, 23, 14, 24, 34, 15, 25, 35, 56, 67, 123, 124, 134, 234, 125, 135, 245, 1234, 1246, 12345, 123456, 123457 and their complements
	\end{itemize}
\end{rem}

\subsection{Chains of size at most $3$}

\noindent $\sats(n,C_3+C_1)\le 8$.
\begin{spverbatim}
	A saturating family of size 8 is: [{}, {1, 2, 3, 4}, {1}, {2, 3, 4}, {2}, {1, 3, 4}, {1, 2}, {3, 4}]
\end{spverbatim}

\bigskip\noindent $\sats(n,2C_3+C_1)\le 28$.
\begin{spverbatim}
	A saturating family of size 28 is: [{}, {1, 2, 3, 4, 5, 6, 7}, {1}, {2, 3, 4, 5, 6, 7}, {2}, {1, 3, 4, 5, 6, 7}, {1, 2}, {3, 4, 5, 6, 7}, {3}, {1, 2, 4, 5, 6, 7}, {1, 3}, {2, 4, 5, 6, 7}, {2, 3}, {1, 4, 5, 6, 7}, {1, 2, 3}, {4, 5, 6, 7}, {4}, {1, 2, 3, 5, 6, 7}, {1, 2, 3, 4}, {5, 6, 7}, {5}, {1, 2, 3, 4, 6, 7}, {1, 2, 3, 5}, {4, 6, 7}, {4, 5}, {1, 2, 3, 6, 7}, {1, 2, 3, 4, 5}, {6, 7}]
\end{spverbatim}

\bigskip\noindent $\sats(n,2C_3+2C_1)\le 26$.
\begin{spverbatim}
	A saturating family of size 26 is: [{}, {1, 2, 3, 4, 5, 6}, {1}, {2, 3, 4, 5, 6}, {2}, {1, 3, 4, 5, 6}, {1, 2}, {3, 4, 5, 6}, {3}, {1, 2, 4, 5, 6}, {1, 3}, {2, 4, 5, 6}, {2, 3}, {1, 4, 5, 6}, {1, 2, 3}, {4, 5, 6}, {4}, {1, 2, 3, 5, 6}, {1, 4}, {2, 3, 5, 6}, {2, 4}, {1, 3, 5, 6}, {1, 2, 4}, {3, 5, 6}, {1, 2, 3, 4}, {5, 6}]
\end{spverbatim}

\bigskip\noindent $\sats(n,2C_3+C_2)\le 20$.
\begin{spverbatim}
	A saturating family of size 20 is: [{}, {1, 2, 3, 4, 5, 6}, {1}, {2, 3, 4, 5, 6}, {2}, {1, 3, 4, 5, 6}, {1, 2}, {3, 4, 5, 6}, {3}, {1, 2, 4, 5, 6}, {1, 3}, {2, 4, 5, 6}, {2, 3}, {1, 4, 5, 6}, {1, 2, 3}, {4, 5, 6}, {4}, {1, 2, 3, 5, 6}, {1, 2, 4}, {3, 5, 6}]
\end{spverbatim}

\bigskip\noindent $\sats(n,3C_3)\le 28$.
\begin{spverbatim}
	A saturating family of size 28 is: [{}, {1, 2, 3, 4, 5, 6, 7}, {1}, {2, 3, 4, 5, 6, 7}, {2}, {1, 3, 4, 5, 6, 7}, {1, 2}, {3, 4, 5, 6, 7}, {3}, {1, 2, 4, 5, 6, 7}, {1, 3}, {2, 4, 5, 6, 7}, {2, 3}, {1, 4, 5, 6, 7}, {1, 2, 3}, {4, 5, 6, 7}, {4}, {1, 2, 3, 5, 6, 7}, {1, 4}, {2, 3, 5, 6, 7}, {1, 2, 4}, {3, 5, 6, 7}, {2, 3, 4}, {1, 5, 6, 7}, {2, 3, 5}, {1, 4, 6, 7}, {1, 4, 5}, {2, 3, 6, 7}]
\end{spverbatim}

\subsection{Chains of size at most $4$}

\noindent $\sats(n,2C_4+C_1)\le 60$.
\begin{spverbatim}
	A saturating family of size 60 is: [{}, {1, 2, 3, 4, 5, 6, 7, 8, 9}, {1}, {2, 3, 4, 5, 6, 7, 8, 9}, {2}, {1, 3, 4, 5, 6, 7, 8, 9}, {1, 2}, {3, 4, 5, 6, 7, 8, 9}, {3}, {1, 2, 4, 5, 6, 7, 8, 9}, {1, 3}, {2, 4, 5, 6, 7, 8, 9}, {2, 3}, {1, 4, 5, 6, 7, 8, 9}, {1, 2, 3}, {4, 5, 6, 7, 8, 9}, {4}, {1, 2, 3, 5, 6, 7, 8, 9}, {1, 4}, {2, 3, 5, 6, 7, 8, 9}, {2, 4}, {1, 3, 5, 6, 7, 8, 9}, {1, 2, 4}, {3, 5, 6, 7, 8, 9}, {3, 4}, {1, 2, 5, 6, 7, 8, 9}, {1, 3, 4}, {2, 5, 6, 7, 8, 9}, {2, 3, 4}, {1, 5, 6, 7, 8, 9}, {1, 2, 3, 4}, {8, 9, 5, 6, 7}, {5}, {1, 2, 3, 4, 6, 7, 8, 9}, {1, 2, 3, 4, 5}, {8, 9, 6, 7}, {6}, {1, 2, 3, 4, 5, 7, 8, 9}, {1, 2, 3, 4, 6}, {8, 9, 5, 7}, {5, 6}, {1, 2, 3, 4, 7, 8, 9}, {1, 2, 3, 4, 5, 6}, {8, 9, 7}, {7}, {1, 2, 3, 4, 5, 6, 8, 9}, {1, 2, 3, 4, 7}, {8, 9, 5, 6}, {5, 7}, {1, 2, 3, 4, 6, 8, 9}, {1, 2, 3, 4, 5, 7}, {8, 9, 6}, {6, 7}, {1, 2, 3, 4, 5, 8, 9}, {1, 2, 3, 4, 6, 7}, {8, 9, 5}, {5, 6, 7}, {1, 2, 3, 4, 8, 9}, {1, 2, 3, 4, 5, 6, 7}, {8, 9}]
\end{spverbatim}

\bigskip\noindent $\sats(n,2C_4+C_2)\le 54$.
\begin{spverbatim}
	A saturating family of size 54 is: [{}, {1, 2, 3, 4, 5, 6, 7, 8, 9}, {1}, {2, 3, 4, 5, 6, 7, 8, 9}, {2}, {1, 3, 4, 5, 6, 7, 8, 9}, {1, 2}, {3, 4, 5, 6, 7, 8, 9}, {3}, {1, 2, 4, 5, 6, 7, 8, 9}, {1, 3}, {2, 4, 5, 6, 7, 8, 9}, {2, 3}, {1, 4, 5, 6, 7, 8, 9}, {1, 2, 3}, {4, 5, 6, 7, 8, 9}, {4}, {1, 2, 3, 5, 6, 7, 8, 9}, {1, 4}, {2, 3, 5, 6, 7, 8, 9}, {2, 4}, {1, 3, 5, 6, 7, 8, 9}, {1, 2, 4}, {3, 5, 6, 7, 8, 9}, {3, 4}, {1, 2, 5, 6, 7, 8, 9}, {1, 3, 4}, {2, 5, 6, 7, 8, 9}, {2, 3, 4}, {1, 5, 6, 7, 8, 9}, {1, 2, 3, 4}, {8, 9, 5, 6, 7}, {5}, {1, 2, 3, 4, 6, 7, 8, 9}, {1, 2, 5}, {1, 2, 6, 7, 8, 9}, {1, 2, 3, 4, 5}, {8, 9, 6, 7}, {6}, {1, 2, 3, 4, 5, 7, 8, 9}, {1, 2, 3, 4, 6}, {8, 9, 5, 7}, {5, 6}, {1, 2, 3, 4, 7, 8, 9}, {1, 2, 3, 4, 5, 6}, {8, 9, 7}, {1, 2, 3, 4, 7}, {8, 9, 5, 6}, {5, 7}, {1, 2, 3, 4, 6, 8, 9}, {1, 2, 3, 4, 6, 7}, {8, 9, 5}, {5, 6, 7}, {1, 2, 3, 4, 8, 9}]
\end{spverbatim}

\bigskip\noindent $\sats(n,2C_4+C_3)\le 38$.
\begin{spverbatim}
	A saturating family of size 38 is:  [{}, {1, 2, 3, 4, 5, 6, 7}, {1}, {2, 3, 4, 5, 6, 7}, {2}, {1, 3, 4, 5, 6, 7}, {1, 2}, {3, 4, 5, 6, 7}, {3}, {1, 2, 4, 5, 6, 7}, {1, 3}, {2, 4, 5, 6, 7}, {2, 3}, {1, 4, 5, 6, 7}, {1, 2, 3}, {4, 5, 6, 7}, {4}, {1, 2, 3, 5, 6, 7}, {1, 4}, {2, 3, 5, 6, 7}, {2, 4}, {1, 3, 5, 6, 7}, {1, 2, 4}, {3, 5, 6, 7}, {3, 4}, {1, 2, 5, 6, 7}, {1, 3, 4}, {2, 5, 6, 7}, {2, 3, 4}, {1, 5, 6, 7}, {1, 2, 3, 4}, {5, 6, 7}, {5}, {1, 2, 3, 4, 6, 7}, {1, 2, 5}, {1, 2, 3, 5}, {1, 2, 6, 7}, {2, 6, 7}]
\end{spverbatim}

\bigskip\noindent $\sats(n,3C_4)\le 52$.
\begin{spverbatim}
	A saturating family of size 52 is: [{}, {1, 2, 3, 4, 5, 6, 7, 8}, {1}, {2, 3, 4, 5, 6, 7, 8}, {2}, {1, 3, 4, 5, 6, 7, 8}, {1, 2}, {3, 4, 5, 6, 7, 8}, {3}, {1, 2, 4, 5, 6, 7, 8}, {1, 3}, {2, 4, 5, 6, 7, 8}, {2, 3}, {1, 4, 5, 6, 7, 8}, {1, 2, 3}, {8, 4, 5, 6, 7}, {4}, {1, 2, 3, 5, 6, 7, 8}, {1, 4}, {2, 3, 5, 6, 7, 8}, {2, 4}, {1, 3, 5, 6, 7, 8}, {1, 2, 4}, {8, 3, 5, 6, 7}, {3, 4}, {1, 2, 5, 6, 7, 8}, {1, 3, 4}, {8, 2, 5, 6, 7}, {2, 3, 4}, {8, 1, 5, 6, 7}, {1, 2, 3, 4}, {8, 5, 6, 7}, {5}, {1, 2, 3, 4, 6, 7, 8}, {1, 5}, {2, 3, 4, 6, 7, 8}, {1, 2, 5}, {8, 3, 4, 6, 7}, {1, 2, 3, 5}, {8, 4, 6, 7}, {3, 4, 5}, {8, 1, 2, 6, 7}, {2, 3, 4, 5}, {8, 1, 6, 7}, {3, 4, 6}, {8, 1, 2, 5, 7}, {2, 3, 4, 6}, {8, 1, 5, 7}, {1, 5, 6}, {8, 2, 3, 4, 7}, {1, 2, 5, 6}, {8, 3, 4, 7}]
\end{spverbatim}

\bigskip\noindent $\sats(n,2C4+2C_2)\le 46$.
\begin{spverbatim}
	A saturating family of size 46 is:  [{}, {1, 2, 3, 4, 5, 6, 7}, {1}, {2, 3, 4, 5, 6, 7}, {2}, {1, 3, 4, 5, 6, 7}, {1, 2}, {3, 4, 5, 6, 7}, {3}, {1, 2, 4, 5, 6, 7}, {1, 3}, {2, 4, 5, 6, 7}, {2, 3}, {1, 4, 5, 6, 7}, {1, 2, 3}, {4, 5, 6, 7}, {4}, {1, 2, 3, 5, 6, 7}, {1, 4}, {2, 3, 5, 6, 7}, {2, 4}, {1, 3, 5, 6, 7}, {1, 2, 4}, {3, 5, 6, 7}, {3, 4}, {1, 2, 5, 6, 7}, {1, 3, 4}, {2, 5, 6, 7}, {2, 3, 4}, {1, 5, 6, 7}, {1, 2, 3, 4}, {5, 6, 7}, {5}, {1, 2, 3, 4, 6, 7}, {1, 5}, {2, 3, 4, 6, 7}, {2, 5}, {1, 3, 4, 6, 7}, {1, 2, 5}, {3, 4, 6, 7}, {1, 2, 3, 5}, {4, 6, 7}, {1, 2, 4, 5}, {3, 6, 7}, {1, 2, 3, 4, 5}, {6, 7}]
\end{spverbatim}

\bigskip\noindent $\sats(n,2C_4+2C_1)\le 68$.
\begin{spverbatim}
	A saturating family of size 68 is: [{}, {1, 2, 3, 4, 5, 6, 7, 8, 9, 10}, {1}, {2, 3, 4, 5, 6, 7, 8, 9, 10}, {2}, {1, 3, 4, 5, 6, 7, 8, 9, 10}, {1, 2}, {3, 4, 5, 6, 7, 8, 9, 10}, {3}, {1, 2, 4, 5, 6, 7, 8, 9, 10}, {1, 3}, {2, 4, 5, 6, 7, 8, 9, 10}, {2, 3}, {1, 4, 5, 6, 7, 8, 9, 10}, {1, 2, 3}, {4, 5, 6, 7, 8, 9, 10}, {4}, {1, 2, 3, 5, 6, 7, 8, 9, 10}, {1, 4}, {2, 3, 5, 6, 7, 8, 9, 10}, {2, 4}, {1, 3, 5, 6, 7, 8, 9, 10}, {1, 2, 4}, {3, 5, 6, 7, 8, 9, 10}, {3, 4}, {1, 2, 5, 6, 7, 8, 9, 10}, {1, 3, 4}, {2, 5, 6, 7, 8, 9, 10}, {2, 3, 4}, {1, 5, 6, 7, 8, 9, 10}, {1, 2, 3, 4}, {5, 6, 7, 8, 9, 10}, {5}, {1, 2, 3, 4, 6, 7, 8, 9, 10}, {1, 2, 5}, {3, 4, 6, 7, 8, 9, 10}, {3, 4, 5}, {1, 2, 6, 7, 8, 9, 10}, {1, 2, 3, 4, 5}, {8, 9, 10, 6, 7}, {6}, {1, 2, 3, 4, 5, 7, 8, 9, 10}, {1, 2, 3, 4, 5, 6}, {8, 9, 10, 7}, {7}, {1, 2, 3, 4, 5, 6, 8, 9, 10}, {1, 2, 3, 4, 5, 7}, {8, 9, 10, 6}, {6, 7}, {1, 2, 3, 4, 5, 8, 9, 10}, {1, 2, 3, 4, 5, 6, 7}, {8, 9, 10}, {8}, {1, 2, 3, 4, 5, 6, 7, 9, 10}, {1, 2, 3, 4, 5, 8}, {9, 10, 6, 7}, {8, 6}, {1, 2, 3, 4, 5, 7, 9, 10}, {1, 2, 3, 4, 5, 6, 8}, {9, 10, 7}, {8, 7}, {1, 2, 3, 4, 5, 6, 9, 10}, {1, 2, 3, 4, 5, 7, 8}, {9, 10, 6}, {8, 6, 7}, {1, 2, 3, 4, 5, 9, 10}, {1, 2, 3, 4, 5, 6, 7, 8}, {9, 10}]
\end{spverbatim}

\end{document}